\documentclass[pdftex,11pt]{amsart} %アーカイブに乗せる時はdvipdfmをpdftexにする
\usepackage[pdftex]{hyperref} %%アーカイブに乗せる時はdvipdfmをpdftexにする
\usepackage{amsmath,amssymb}
\usepackage{mathtools}
\usepackage{amscd}
\usepackage{mathrsfs}
\usepackage[foot]{amsaddr}
\usepackage{amsthm}
\usepackage{setspace}
\usepackage[all]{xy}
\usepackage{comment}
\usepackage[pdftex]{graphicx} %%アーカイブに乗せる時はdvipdfmをpdftexにする
\usepackage{tikz}
\usetikzlibrary{arrows}
\usetikzlibrary{shapes.geometric,fit}
\usepackage{tikz-cd}
\usepackage{fancyhdr}

\theoremstyle{definition}
\newtheorem{thm}{Theorem}[subsection]
\newtheorem*{thm*}{Theorem}
\newtheorem{defi}[thm]{Definition}
\newtheorem*{defi*}{Definition}
\newtheorem*{acknowledge}{Acknowledgement}
\newtheorem{cor}[thm]{Corollary}
\newtheorem*{cor*}{Corollary}
\newtheorem{prop}[thm]{Proposition}
\newtheorem*{prop*}{Proposition}
\newtheorem{lem}[thm]{Lemma}
\newtheorem*{lem*}{Lemma}
\newtheorem{ex}[thm]{Example}
\newtheorem*{ex*}{Example}

\newtheorem{rem}[thm]{Remark}
\newtheorem*{rem*}{Remark}

\newtheorem*{hw*}{Homework}

\newcommand{\Z}{\mathbb{Z}}
\newcommand{\N}{\mathbb{N}}
\newcommand{\T}{\mathbb{T}}

\newcommand{\tight}{\mathrm{tight}}

\DeclareMathOperator{\dom}{\mathrm{dom}}

\def\i<#1>{\langle #1 \rangle}
\def\l<#1>{\left\langle #1 \right\rangle}

\makeatletter
\renewenvironment{proof}[1][\proofname]{\par
  \normalfont
  \topsep6\p@\@plus6\p@ \trivlist
  \item[\hskip\labelsep{\bfseries #1}\@addpunct{.}]\ignorespaces
}{%
  \endtrivlist
}
\renewcommand{\proofname}{\sc{Proof}}
\makeatother
\makeatletter
\newcommand*{\defeq}{\mathrel{\rlap{%
                     \raisebox{0.3ex}{$\m@th\cdot$}}%
                     \raisebox{-0.3ex}{$\m@th\cdot$}}%
                     =}
\makeatother

\title[]{A correspondence between inverse subsemigroups, open wide subgroupoids and Cartan intermediate C*-subalgebras}
\author{Fuyuta Komura}
\address{Department of Mathematics, Faculty of Science and Technology, Keio University,
	3-14-1 Hiyoshi, Kohoku-ku, Yokohama, 223-8522, Japan}
\address{ Mathematical Science Team
Center for Advanced Intelligence Project (AIP)
RIKEN}
\address{Phone: +81-45-566-1641+42706}
\address{Fax: +81-45-566-1642}
\email{fuyuta.k@keio.jp}

\begin{document}
\maketitle
\begin{abstract}

For a given inverse semigroup action on a topological space,
one can associate an \'etale groupoid. 
We prove that there exists a correspondence between the certain subsemigroups and the open wide subgroupoids in case that the action is strongly tight.
Combining with the recent result of Brown et.\ al,
we obtain a correspondence between the certain subsemigroups of an inverse semigroup and the Cartan intermediate subalgebras of a groupoid C*-algebra. 
\end{abstract}

\section{Introduction}

Given an action of an inverse semigroup on a topological space,
one can associate an \'etale groupoid by taking a germ.
For a given \'etale groupoid,
we can construct groupoid C*-algebras, which are initiated by Renault \cite{renault1980groupoid}.
It is a natural task to investigate the relation among them and actually many researchers have been doing this.
In this paper,
we establish a correspondence between the set of certain subsemigroups and the set of wide open wide subgroupoids of the associated groupoids.
We consider inverse semigroups acting on topological spaces in the ``strongly tight'' way (see Definition \ref{definition of strongly tight}).
Our main theorem, Theorem \ref{main theorem}, states that wide open subgroupoids of associated groupoids with strongly tight actions corresponds to certain subsemigroups of the inverse semigroups.
Combining with the work in \cite{Brown:2019aa},
we obtain a correspondence between Cartan intermediate subalgebras in groupoid C*-algebras and certain subsemigroups of inverse semigroups.
As an application,
we compute all Cartan intermediate subalgebras of the Cuntz algebras which contains the fixed point algebras. 

This paper is organized as follows.
Section 1 is devoted for preliminaries.
In Section 2,
we investigate open subgroupoids of \'etale groupoids associated to strongly tight actions.
Then we establish a correspondence between open wide subgroupoids and certain subsemigroups (Theorem \ref{main theorem}).

In Section 3,
we give applications of our correspondence.
The first application is regarding with inverse semigroups which consist of compact bisections of \'etale groupoids.
We show that a class of open wide subgroupoids of an ample groupoid is described by an inverse semigroup of compact bisections (Corollary \ref{Corollary about inverse semigroups of compact bisections}).
As the second application,
we study certain subsemigroups of the polycyclic monoids.
This study is applied to the computation of Cartan intermediate subalgebras between the Cuntz algebras and the fixed point algebras.

In Section 4,
we summarize the relation between Cartan intermediate subalgebras of C*-algebras and certain subsemigroups of inverse semigroups.
Then we compute Cartan intermediate subalgebras of the Cuntz algebras which contains the fixed point algebras.

In Section 5,
we mention the relation between strongly tight actions and tight groupoids.
We give a characterization of a tight groupoid with the compact unit space in Corollary \ref{cor about compact strongly tight action}.

\begin{acknowledge}
	The author would like to thank Prof.\ Takeshi Katsura for his support and encouragement.
	This work was supported by JSPS KAKENHI 20J10088.
\end{acknowledge}

\section{Preliminaries}

\subsection{Inverse semigroups}

We recall the basic notions about inverse semigroups.
See \cite{lawson1998inverse} or \cite{paterson2012groupoids} for more details.
An inverse semigroup $S$ is a semigroup where for every $s\in S$ there exists a unique $s^*\in S$ such that $s=ss^*s$ and $s^*=s^*ss^*$.
We denote the set of all idempotents in $S$ by $E(S)\defeq\{e\in S\mid e^2=2\}$.
It is known that $E(S)$ is a commutative subsemigroup of $S$.
An inverse semigroup which consists of idempotents is called a (meet) semilattice.
A zero element is a unique element $0\in S$ such that $0s=s0=0$ holds for all $s\in S$.
A unit is a unique element $1\in S$ such that $1s=s1=s$ holds for all $s\in S$.
In this paper, \textbf{we assume that every inverse semigroup always has a zero element},
although it does not necessarily have a unit.
An inverse semigroup with a unit is called an inverse monoid.
By a subsemigroup of $S$,
we mean a subset of $S$ which is closed under the product and inverse of $S$.
For $s,t\in S$,
we write $s\leq t$ if $ts^*s=s$ holds.
Then this defines a partial order on $S$.
Note that $e\leq f$ holds if and only if $ef=e$ holds for $e,f\in E(S)$.
A pair $s,t\in S$ is said to be compatible if $s^*t, st^*\in E(S)$ holds.
Notice that $s,t$ are compatible if there exists $u\in S$ such that $s,t\leq u$.
A subsemigroup of an inverse semigroup $S$ is said to be wide if it contains $E(S)$.
A subset $I\subset E(S)$ is called an ideal if $e\in I$ and $f\leq e$ implies $f\in I$.
A subset $C\subset I$ of an ideal $I\subset E(S)$ is called a cover if for every $e\in I\setminus \{0\}$ there exists $c\in C$ such that $ec\not=0$.

For a topological space $X$,
we denote by $I_X$ the set of all homeomorphisms between open sets in $X$.
Then $I_X$ is an inverse semigroup with respect to the product defined by the composition of maps.
For $f,g\in I_X$,
note that $f\leq g$ holds if and only if $\dom f\subset \dom g$ and $f(x)=g(x)$ hold for all $x\in \dom f$.

\begin{comment}
Assume that $f,g\in I_X$ are compatible.
Then the map $h$ defined by
\begin{align*}
h(x)\defeq 
\begin{cases}
f(x) & (x\in \dom f) \\
g(x) & (x\in \dom g \setminus \dom f)
\end{cases}
\end{align*}
belongs to $I_X$.
Note that $\dom h=\dom f\cup \dom g$ and $h$ is the least upper bound of $f$ and $g$.
Thus we obtain the join $f\vee g=h$ of a compatible pair $f,g\in I_X$.
A subsemigroup $T\subset I_X$ is said to be join closed if the joins of all compatible pairs of $T$ also belong to $T$.
\end{comment}

\subsection{\'Etale groupoids}

We recall the basic notions on \'etale groupoids.
See \cite{asims} and \cite{paterson2012groupoids} for more details.

A groupoid is a set $G$ together with a distinguished subset $G^{(0)}\subset G$,
 source and range maps $d,r\colon G\to G^{(0)}$ and a multiplication 
\[
G^{(2)}\defeq \{(\alpha,\beta)\in G\times G\mid d(\alpha)=r(\beta)\}\ni (\alpha,\beta)\mapsto \alpha\beta \in G
\]
such that
\begin{enumerate}
	\item for all $x\in G^{(0)}$, $d(x)=x$ and $r(x)=x$ hold,
	\item for all $\alpha\in G$, $\alpha d(\alpha)=r(\alpha)\alpha=\alpha$ holds,
	\item for all $(\alpha,\beta)\in G^{(2)}$, $d(\alpha\beta)=d(\beta)$ and $r(\alpha\beta)=r(\alpha)$ hold,
	\item if $(\alpha,\beta),(\beta,\gamma)\in G^{(2)}$,
	we have $(\alpha\beta)\gamma=\alpha(\beta\gamma)$,
	\item\label{inverse} every $\gamma \in G$,
	there exists $\gamma'\in G$ which satisfies $(\gamma',\gamma), (\gamma,\gamma')\in G^{(2)}$ and $d(\gamma)=\gamma'\gamma$ and $r(\gamma)=\gamma\gamma'$.   
\end{enumerate}
Since the element $\gamma'$ in (\ref{inverse}) is uniquely determined by $\gamma$,
$\gamma'$ is called the inverse of $\gamma$ and denoted by $\gamma^{-1}$.
We call $G^{(0)}$ the unit space of $G$.
A subgroupoid of $G$ is a subset of $G$ which is closed under the inversion and multiplication. 
A subgroupoid of $G$ is said to be wide if it contains $G^{(0)}$.

%For $U\subset G^{(0)}$, we define $G_U\defeq d^{-1}(U)$ and $G^{U}\defeq r^{-1}(U)$.
%We define also $G_x\defeq G_{\{x\}}$ and $G^x\defeq G^{\{x\}}$ for $x\in G^{(0)}$.
%The isotropy bundle of $G$ is denoted by $\Iso(G)\defeq\{\gamma\in G\mid d(\gamma)=r(\gamma)\}$.
%A subset $F\subset G^{(0)}$ is said to be invariant if $d(\alpha)\in F$ implies $r(\alpha)\in F$ for all $\alpha\in G$.
%たぶん不変集合はいらんよね？？

A topological groupoid is a groupoid equipped with a topology where the multiplication and the inverse are continuous.
A topological groupoid is said to be \'etale if the source map is a local homeomorphism.
Note that the range map of an \'etale groupoid is also a local homeomorphism.
An \'etale groupoid is said to be ample if it has an open basis which consists of compact sets.
In this paper,
we mainly treat ample groupoids.

A topological groupoid $G$ is said to be topologically principal if the set
\[
\{x\in G^{(0)}\mid G(x)=\{x\}\}
\]
is dense in $G^{(0)}$,
where $G(x)$ is the isotropy group at $x\in G^{(0)}$ :
\[
G(x)\defeq \{\alpha\in G\mid d(\alpha)=r(\alpha)=x\}.
\]
\subsection{\'Etale groupoids associated to inverse semigroup actions}

An \'etale groupoid arises from an action of an inverse semigroup to a topological space.
We recall how to construct an \'etale groupoid from an inverse semigroup action. 
We begin with the definition of an inverse semigroup action.

Let $X$ be a topological space.
Recall that $I_X$ is an inverse semigroup of homeomorphisms between open sets in $X$.
An action $\alpha\colon S\curvearrowright X$ is a semigroup homomorphism $S\ni s\mapsto \alpha_s\in I_X$.
In this paper, we always assume that every action $\alpha$ satisfies $\bigcup_{e\in E(S)}\dom(\alpha_e)=X$ and $\dom(\alpha_0)=\emptyset$. 
For $e\in E(S)$, we denote the domain of $\alpha_e$ by $D_e^{\alpha}$.
Then $\alpha_s$ is a homeomorphism from $D_{s^*s}^{\alpha}$ to $D_{ss^*}^{\alpha}$.
We often omit $\alpha$ of $D_{e}^{\alpha}$ if there is no chance to confuse.

For an action $\alpha\colon S\curvearrowright X$,
we associate an \'etale groupoid $S\ltimes_{\alpha}X$ as the following.
First we put the set $S*X\defeq \{(s,x) \in S\times X \mid x\in D^{\alpha}_{s^*s}\}$.
Then we define an equivalence relation $\sim$ on $S*X$ by declaring that $(s,x)\sim (t,y)$ holds if
\[
\text{$x=y$ and there exists $e\in E(S)$ such that $x\in D^{\alpha}_e$ and $se=te$}.  
\]
Set $S\ltimes_{\alpha}X\defeq S*X/{\sim}$ and denote the equivalence class of $(s,x)\in S*X$ by $[s,x]$.
The unit space of $S\ltimes_{\alpha}X$ is $X$, where $X$ is identified with the subset of $S\ltimes_{\alpha}X$ via the injection
\[
X\ni x\mapsto [e,x] \in S\ltimes_{\alpha}X, x\in D^{\alpha}_e.
\]
The source map and range maps are defined by
\[
d([s,x])=x, r([s,x])=\alpha_s(x)
\]
for $[s,x]\in S\ltimes_{\alpha}X$.
The product of $[s,\alpha_t(x)],[t,x]\in S\ltimes_{\alpha}X$ is $[st,x]$.
The inverse should be $[s,x]^{-1}=[s^*,\alpha_s(x)]$.
Then $S\ltimes_{\alpha}X$ is a groupoid in these operations.
For $s\in S$ and an open set $U\subset D_{s^*s}^{\alpha}$,
define 
\[[s, U]\defeq \{[s,x]\in S\ltimes_{\alpha}X\mid x\in U\}.\]
These sets form an open basis of $S\ltimes_{\alpha}X$.
In these structures,
$S\ltimes_{\alpha}X$ is an \'etale groupoid.

\section{Correspondence between subsemigroups and subgroupoids}

In this section,
we consider strong tight actions of inverse semigroups (Definition \ref{definition of strongly tight}).
Then we establish a correspondence between certain subsemigroups of an inverse semigroup and open wide subgroupoids of an \'etale groupoid associated to a strongly tight action (Theorem \ref{main theorem}).
Then we observe a condition for an open wide subgroupoid to be closed in terms of an inverse semigroup.

\subsection{Correspondence between subsemigroups and subgroupoids}

We begin with the definition of a strongly tight action.

\begin{defi}\label{definition of strongly tight}
	Let $S$ be an inverse semigroup and $X$ be a locally compact Hausdorff space.
	An action $\alpha\colon S\curvearrowright X$ is said to be ample if $D_e^{\alpha}\subset X$ is a compact set for all $e\in E(S)$.
	We say that an action $\alpha	\colon S\curvearrowright X$ is strongly tight if $\{D_e^{\alpha}\}_{e\in E(S)}$ is a basis of $X$.
\end{defi}
	
	We remark that if there exists a strongly tight action $\alpha\colon S\curvearrowright X$, then $X$ is totally disconnected and $S\ltimes_{\alpha}X$ is an ample groupoid.

	Strongly tight actions are related with the actions on tight spectrums of inverse semigroups,
	which are investigated in \cite{ExelPardo2016}.
	We will see a relation between strongly tight actions and tight groupoids in Section \ref{section about characterization of strongly tight action}.
	
	We construct subsemigroups from wide groupoids.

%%%%%%%%%%%%%%%7/17 ここまでやった
%%%%%%%%%%%%%%%%%%%

\begin{defi}
	Let $S$ be an inverse semigroup,
	$X$ be a locally compact Hausdorff space
	and $\alpha\colon S\curvearrowright X$ be an action.
	Put $G\defeq S\ltimes_{\alpha}X$.
	For a wide subgroupoid $H\subset G$,
	we define 
	\[
	T_H\defeq \{s\in S\mid [s,D_{s^*s}]\subset H\}.
	\]
\end{defi}
\begin{prop}
	In the above notations,
	$T_H$ is a wide subsemigroup of $S$.
\end{prop}

\begin{proof}
	For $e\in E(S)$,
	$[e,D_e]\in G^{(0)}\subset H$ holds.
	Hence $T_H$ contains $E(S)$.
	
	Next we show that $T_H$ is an subsemigroup of $S$.
	We show $st\in T_H$ for $s,t\in T_H$.
	For $x\in D_{(st)^*st}$,
	it follows that $[s,\alpha_t(x)],[t,x]\in H$ from $s,t\in T_H$.
	Thus we obtain
	\[
	[st,x]=[s,\alpha_t(x)][t,x]\in H.
	\]
	Therefore we have $[st,D_{(st)^*st}]\subset H$ and $st\in T_H$.
	
	It is clear that $T_H$ is closed under the inverse.
	Hence $T_H$ is a wide subsemigroup of $S$.
	\qed
\end{proof}

We define a class of subsemigroups which corresponds to open wide subgroupoids (c.f.\ Theorem \ref{main theorem}).

\begin{defi}
		Let $S$ be an inverse semigroup,
		$X$ be a locally compact Hausdorff space
		and $\alpha\colon S\curvearrowright X$ be an action.
		A wide subsemigroup $T\subset S$ is said to be $\alpha$-join closed if $T$ has the next property : 
		\begin{center}
		`For every $s\in S$,
		$s$ belongs to $T$ if and only if there exists a finite set $F\subset E(S)$ such that $sf\in T$ holds for all $f\in F$ and $D_{s^*s}\subset\bigcup_{f\in F}D_{f}$ holds.'
		\end{center}
\end{defi}
\begin{rem}
	The ``only if'' part in the previous definition always holds for all wide subsemigroups $T$.
\end{rem}

\begin{prop}
	Let $S$ be an inverse semigroup,
	$X$ be a locally compact Hausdorff space
	and $\alpha\colon S\curvearrowright X$ be an action.
	For a wide subgroupoid $H\subset S\ltimes_{\alpha}X$,
	a wide subsemigroup $T_H\subset S$ is $\alpha$-join closed. 
\end{prop}

\begin{proof}
	Take $s\in S$ and assume that there exists a finite set $F\subset E(S)$ such that $sf\in T_H$ for all $f\in F$ and $D_{s^*s}\subset \bigcup_{f\in F}D_f$.
	It suffices to show $s\in T_H$.
	For $x\in D_{s^*s}$,
	there exists $f\in F$ with $x\in D_f$.
	Since we have $sf\in T_H$,
	it follows
	\[
	[s,x]=[sf,x]\in H.
	\]
	Thus we obtain $[s,D_{s^*s}]\subset H$ and therefore $s\in T_H$.
	\qed
\end{proof}

The proof of the next proposition is left to the reader.

\begin{prop}
		Let $S$ be an inverse semigroup,
		$X$ be a locally compact Hausdorff space
		and $\alpha\colon S\curvearrowright X$ be an action.
		For a wide subsemigroup $T\subset S$,
		the map
		\[
		T\ltimes_{\alpha}X\ni [t,x]\mapsto [t,x]\in S\ltimes_{\alpha}X
		\]
		is an open map and an isomorphism onto its image.
\end{prop}

Via the map in the previous proposition,
$T\ltimes_{\alpha}X$ is identified with the wide open subgroupoid of $S\ltimes_{\alpha}X$.

\begin{lem}\label{lemma 1 of main theorem}
		Let $S$ be an inverse semigroup,
	$X$ be a locally compact Hausdorff space
	and $\alpha\colon S\curvearrowright X$ be an action.
	For a wide subsemigroup $T\subset S$,
	$T_{T\ltimes_{\alpha} X}\supset T$ holds.
	Moreover,
	if $T$ is $\alpha$-join closed,
	$T_{T\ltimes_{\alpha} X}= T$ holds.
\end{lem}

\begin{proof}
	The inclusion $T_{T\ltimes_{\alpha} X}\supset T$ is clear.
	We assume that $T$ is $\alpha$-join closed and show $T_{T\ltimes_{\alpha} X}\subset T$.
	Take $s\in T_{T\ltimes_{\alpha}X}$ and fix $x\in D_{s^*s}$.
	Since we have $[s,x]\in T\ltimes_{\alpha}X$,
	there exists $e_x\in E(S)$ such that $se_x\in T$ and $x\in D_{e_x}$.
	Since we assume that $D_{s^*s}$ is compact,
	there exists a finite set $P\subset D_{s^*s}$ with $D_{t^*t}\subset \bigcup_{x\in P}D_{e_x}$.
	Using the condition that $T$ is $\alpha$-join closed,
	we obtain $t\in T$.
	Now we have shown $T_{T\ltimes_{\alpha} X}\subset T$.
	\qed
\end{proof}

\begin{lem}\label{lemma 2 of main theorem}
	Let $S$ be an inverse semigroup,
	$X$ be a locally compact Hausdorff space
	and $\alpha\colon S\curvearrowright X$ be an ample action.
	Put $G\defeq S\ltimes_{\alpha}X$.
	For a wide groupoid $H\subset G$,
	$T_H\ltimes_{\alpha}X\subset H$ holds.
	Moreover, if $H\subset G$ is open and $\alpha\colon S\curvearrowright X$ is strongly tight,
	$T_H\ltimes_{\alpha}X= H$ also holds.
\end{lem}

\begin{proof}
	Assume that $[s,x]\in T_H\ltimes_{\alpha}X$.
	Then there exists $t\in T_H$ such that $[s,x]=[t,x]$.
	Now we have 
	\[[s,x]=[t,x]\in [t,D_{t^*t}]\subset H. 
	\]
	
	Next we show the other inclusion $T_H\ltimes_{\alpha}X\supset H$ under the assumption that $\alpha$ is strongly tight and $H$ is open.
	Take $[s,x]\in H$.
	Since $H$ is open and $\alpha$ is strongly tight,
	there exists $e\in E(S)$ such that $x\in D_e\subset D_{s^*s}$ and $[s,D_e]\subset H$.
	One can see $[se,D_{(se)^*se}]\subset H$,
	so we have $se\in T_H$.
	Therefore it follows $[s,x]=[se,x]\in T_H\ltimes_{\alpha}X$.
	\qed
\end{proof}

The next theorem follows from Lemma \ref{lemma 1 of main theorem} and Lemma \ref{lemma 2 of main theorem}.

\begin{thm}\label{main theorem}
	Let $S$ be an inverse semigroup,
	$X$ be a locally compact Hausdorff space
	and $\alpha\colon S\curvearrowright X$ be an action.
	Assume that $\alpha\colon S\curvearrowright X$ is strongly tight.
	Let $\mathcal{T}$ denote the set of all wide $\alpha$-join closed subsemigroups of $S$.
	In addition, let $\mathcal{H}$ denote the set of all wide open subgroupoids of $G$.
	Then  maps \[\mathcal{T}\ni T\to T\ltimes_{\alpha}X\in \mathcal{H}\] and \[\mathcal{H}\ni H\to T_H\in \mathcal{T}\] are inverse maps of each other.
\end{thm}

%%%%%%%%%%%%7/17ここまでやった
%%%%%%%%%%%%%%%%%%%

\subsection{Closedness of subgroupoid}

We give conditions where $T\ltimes_{\alpha}X$ is closed in $S\ltimes_{\alpha}X$.

\begin{defi}
	Let $S$ be an inverse semigroup and $T\subset S$ be a wide subsemigroup.
	For $s\in S$, we define $\mathcal{J}_s^{T}\subset E(S)$ as
	\[
	\mathcal{J}_s^{T}\defeq \{e\in E(S)\mid se\in T \text{ and } e\leq s^*s\}.
	\]
\end{defi}

\begin{prop}
	In the above notations,
	$\mathcal{J}_s^{T}$ is an ideal of $E(S)$.
\end{prop}
\begin{proof}
	Assume $e\in\mathcal{J}_s^T$ and $f\leq e$.
	Then we have $sf=sef\in T$.
	Hence we obtain $f\in \mathcal{J}_s^T$.
	\qed
\end{proof}

We remark that 
\[
\mathcal{J}_s^{E(S)}=\{e\in E(S)\mid e\leq s\}
\]
holds.
This ideal appears in \cite[Definition 3.11]{ExelPardo2016}.

Assume that an action $\alpha\colon S\curvearrowright X$ is given.
For an ideal $\mathcal{J}\subset E(S)$,
we define $D(\mathcal{J})\defeq\bigcup_{e\in\mathcal{J}}D_e$.
The next lemma is a slight generalization of \cite[Proposition 3.14]{ExelPardo2016}.

\begin{lem}\label{formula for the domain of ideal}
		Let $S$ be an inverse semigroup,
		$X$ be a locally compact Hausdorff space
		and $\alpha\colon S\curvearrowright X$ be an action.
		Assume that we are given a wide subsemigroup $T\subset S$.
		Then the formula
		\[
		[s,D(\mathcal{J}^T_s)]=[s,D_{s^*s}]\cap (T\ltimes_{\alpha}X)
		\]
		holds for all $s\in S$.
\end{lem}

\begin{proof}
	Take $[s,x]\in [s, D(\mathcal{J}^{T}_s)]$.
	Then there exists $e\in\mathcal{J}^T_s$ with $x\in D_e$.
	By the definition of $\mathcal{J}^T_s$,
	we have $se\in T$ and $e\leq s^*s$.
	Hence we obtain
	\[
	[s,x]=[se,x]\in [s,D_{s^*s}]\cap (T\ltimes_{\alpha}X).
	\]
	Now we have shown $[s,D(\mathcal{J}^T_s)]\subset [s,D_{s^*s}]\cap (T\ltimes_{\alpha}X)$.
	To show the reverse inclusion,
	take $[s,x] \in [s,D_{s^*s}]\cap (T\ltimes_{\alpha}X)$.
	Since $[s,x]$ belongs to $T\ltimes_{\alpha}X$,
	there exists $t\in T$ and $f\in E(S)$ such that $sf=tf$ and $x\in D_f$ hold.
	Since we have $ss^*sf=sf=tf\in T$,
	$s^*sf$ belongs to $\mathcal{J}^{T}_s$.
	Since we also have $x\in D_{s^*sf}\subset \mathcal{J}^{T}_s$,
	we obtain $[s,x] \in [s,D(\mathcal{J}_s^T)]$.
	\qed
\end{proof}

%%%%%%%%%%%%%%7/17ここまで読んだ

\begin{prop}\label{closedness of subgroupoid}
	Let $S$ be an inverse semigroup,
	$X$ be a locally compact Hausdorff space
	and $\alpha\colon S\curvearrowright X$ be an action.
	Assume that we are given a wide subsemigroup $T\subset S$.
	The following conditions are equivalent :
	\begin{enumerate}
	\item $T\ltimes_{\alpha}X$ is a closed subset of $S\ltimes_{\alpha}X$,
	\item for every $s\in S$, $D(\mathcal{J}^T_s)$ is a closed subset of $D_{s^*s}$ with respect to the relative topology of $D_{s^*s}$.
	\end{enumerate}
\end{prop}

\begin{proof}
	First we show that (1) implies (2).
	By Lemma \ref{formula for the domain of ideal} and (1),
	$[s,D(\mathcal{J}_s^T)]$ is a closed subset of $[s,D_{s^*s}]$.
	Since the restriction of the domain map $d\colon[s,D_{s^*s}]\to D_{s^*s}$ is a homeomorphism,
	$d([s,D(\mathcal{J}_s^T)])=D(\mathcal{J}_s^T)$ is closed in $D_{s^*s}$.
	Next we show that (2) implies (1).
	It follows that $[s,D(\mathcal{J}^T_s)]$ is a closed subset of $[s,D_{s^*s}]$ from the same argument in the above and (2).
	We have that $[s,D_{s^*s}]\setminus [s,D(\mathcal{J}^T_s)]$ is open in $S\ltimes_{\alpha}X$ since $[s,D_{s^*s}]$ is open in $S\ltimes_{\alpha}X$ and $[s,D(\mathcal{J}^{T}_s)]$ is closed in $[s,D_{s^*s}]$.
	One can see that the formula
	\[
	S\ltimes_{\alpha}X\setminus T\ltimes_{\alpha}X=\bigcup_{s\in S}([s,D_{s^*s}]\setminus [s,D(\mathcal{J}^T_s)])
	\]
	holds.
	Hence $S\ltimes_{\alpha}X\setminus T\ltimes_{\alpha}X$ is open in $S\ltimes_{\alpha}X$, which implies $T\ltimes_{\alpha}X$ is closed in $S\ltimes_{\alpha}X$.
	\qed
\end{proof}

The next Lemma is essentially same as the \cite[Proposition 3.7]{ExelPardo2016}.
We give a proof for the reader's convenience.

%%%%%%%%%%%%%%%%%%%7/17 ここまで

\begin{lem}[c.f.\ {\cite[Proposition 3.7]{ExelPardo2016}}] \label{lemma about cover}
	Let $S$ be an inverse semigroup,
	$X$ be a locally compact Hausdorff space
	and $\alpha\colon S\curvearrowright X$ be a strongly tight action.
	Assume that $D_e\not=\emptyset$ holds for every $e\in E(S)\setminus\{0\}$.
	For a ideal $\mathcal{J}\subset E(S)$ and a subset $C\subset \mathcal{J}$,
	the followings are equivalent :
	\begin{enumerate}
		\item $C$ is a cover of $\mathcal{J}$,
		\item $\bigcup_{c\in C}D_c=D(\mathcal{J})$ holds.
	\end{enumerate}
\end{lem}

\begin{proof}
	First we show (1) implies (2).
	The inclusion $\bigcup_{c\in C}D_c\subset D(\mathcal{J})$ follows from $C\subset\mathcal{J}$.
	We show the reverse inclusion.
	Take $x\in D(\mathcal{J})$.
	Assume that $x\not\in D_c$ holds for all $c\in C$.
	For each $c\in C$, there exists $e_c\in E(S)$ such that $x\in D_{e_c}$ and $D_{e_c}\cap D_c=\emptyset$ since each $D_c$ is closed in $X$ and $\{D_e\}_{e\in E(S)}$ is a basis of $X$.
	Since $D_{ce_c}=D_{c}\cap D_{e_c}=\emptyset$ and we assume $D_e\not=\emptyset$ for all $e\in E(S)\setminus\{0\}$,
	we have $ce_c=0$.
	Putting $p\defeq \prod_{c\in C}e_c$, we have  $p\in\mathcal{J}\setminus\{0\}$ since $\mathcal{J}$ is ideal and $x\in D_p$.
	However we also have $cp=0$ for each $c\in C$, which contradicts to the condition that $C$ is a cover.
	
	Next we show (2) implies (1).
	Take $e\in \mathcal{J}\setminus\{0\}$.
	Then there exists $c\in C$ such that $D_e\cap D_c\not=\emptyset$,
	which implies $ec\not=0$.
	Hence $C$ is a cover of $\mathcal{J}$.
	\qed
\end{proof}

Now we obtain the characterization about the closedness of open wide subgroupoids.

\begin{thm}
	Let $S$ be an inverse semigroup,
	$X$ be a locally compact Hausdorff space
	and $\alpha\colon S\curvearrowright X$ be a strongly tight action.
	Assume that $D_e\not=\emptyset$ holds for every $e\in E(S)\setminus\{0\}$.
	For a wide subsemigroup $T\subset S$, the following conditions are equivalent :
	\begin{enumerate}
		\item $T\ltimes_{\alpha}X$ is closed in $S\ltimes_{\alpha}X$,
		\item for every $s\in S$, $D(\mathcal{J}^T_s)$ is relatively closed in $D_{s^*s}$,
		\item for every $s\in S$,
		$\mathcal{J}^T_s$ has a finite cover.
	\end{enumerate}
\end{thm}

\begin{proof}
Now it suffices to show that (2) and (3) are equivalent,
since Proposition \ref{closedness of subgroupoid} states that (1) and (2) are equivalent.
First we show that (2) implies (3).
Since we assume that the action $\alpha$ is ample,
$D_{s^*s}$ is compact.
Then $D(\mathcal{J}_s^T)$ is also compact by (2).
Hence there exists a finite set $C\subset\mathcal{J}_s^T$ such that $\bigcup_{c\in C}D_c=D(\mathcal{J}_s^T)$.
By Lemma \ref{lemma about cover},
$C$ is a finite cover of $\mathcal{J}^T_s$.

Next we show that (3) implies (2).
Take $s\in S$ and a finite cover $C$ of $\mathcal{J}^T_s$.
By Lemma \ref{lemma about cover} again,
we have $D(\mathcal{J})=\bigcup_{c\in C}D_c$.
Hence we have $D(\mathcal{J})$ is compact and therefore closed in $D_{s^*s}$ since each $D_c$ is compact.
\qed
\end{proof}

Wide clopen subgroupoids arise from partial group homomorphisms.
We observe this fact in the remainder of this subsection.
 
Let $S$ be an inverse semigroup and $\Gamma$ be a group.
Put $S^{\times}\defeq S\setminus\{0\}$.
A map $\sigma\colon S^{\times}\to\Gamma$ is called a partial homomorphism if $\sigma(st)=\sigma(s)\sigma(t)$ holds for any pair $s,t\in S^{\times}$ with $st\not=0$. 
A partial homomorphism gives us a suitable subsemigroup as follows.

\begin{prop}
	Let $S$ be an inverse semigroup, $\Gamma$ be a group and $\sigma\colon S^{\times}\to \Gamma$ be a partial homomorphism.
	Assume that we are given a locally compact space $X$ and an action $\alpha\colon S\curvearrowright X$ where $D_e\not=\emptyset$ holds for each $e\in E(S)\setminus\{0\}$.
	Then the following statements hold :
	\begin{enumerate}
	\item $\ker\sigma\defeq \sigma^{-1}(e)\cup\{0\}$ is a $\alpha$-join closed wide subsemigroup of $S$, 
	\item $\ker\sigma\ltimes_{\alpha}X$ is closed in $S\ltimes_{\alpha}X$.
	\end{enumerate}
\end{prop}
\begin{proof}
	First we show (1).
	One can see that $\ker\sigma$ is a wide subsemigroup of $S$ in a straightforward way.
	We show $\ker\sigma$ is $\alpha$-join closed.
	Take $s\in S$ and assume that there exists a finite set $F\subset E(S)$ such that $sF\subset\ker\sigma$ and $D_{s^*s}\subset \bigcup_{f\in F}D_f$.
	It suffices to show $s\in\ker\sigma$.
	We may assume that $s\not=0$.
	Then there exists $f\in F$ such that $D_{s^*s}\cap D_f\not=\emptyset$,
	which implies $sf\not=0$.
	Since we have $sf\in\ker\sigma$,
	it follows
	\[
	\sigma(s)=\sigma(s)e=\sigma(s)\sigma(f)=\sigma(sf)=e.
	\]
	Hence $s\in\ker\sigma$.
	
	Next we show (2).
	Although it is possible to apply Proposition \ref{closedness of subgroupoid},
 	we show (2) using a cocycle\footnote{A map from a groupoid to a group is called a cocycle if it preserves the products.} on a groupoid.
 	We define the map $c_{\sigma}\colon S\ltimes_{\alpha}X\to \Gamma$ by
 	\[
 	c_{\sigma}([s,x])=\sigma(s), [s,x]\in S\ltimes_{\alpha}X.
 	\]
 	Then $c_{\sigma}$ is a continuous cocycle.
 	One can see that
 	\[
 	\ker\sigma\ltimes_{\alpha}X= c_{\sigma}^{-1}(e)
 	\]
 	holds.
 	Hence $\ker\sigma\ltimes_{\alpha}X$ is closed in $S\ltimes_{\alpha}X$.
 	\qed
 	
\end{proof}

%%%%%%%%%%7/20 ここまでよんだ
%%%%%%%%%%%

\section{Applications and examples}

\subsection{Inverse semigroups of compact bisections}

Let $G$ be an ample \'etale groupoid.
Recall that an open set $U\subset G$ is called a bisection if the restrictions $d|_U$ and $r|_U$ are homeomorphisms onto the images.
Let $I(G)$ denote the set of all compact bisections of $G$.
For $U,V\in I(G)$,
their product $UV$ is defined by
\[
UV\defeq\{\alpha\beta\in G\mid \alpha\in U, \beta\in V, d(\alpha)=r(\beta)\}.
\]
Then $UV$ belongs to $I(G)$.
It is known that $I(G)$ becomes an inverse semigroup.
Note that the inverse of $U\in I(G)$ is given by
\[
U^{-1}=\{\alpha^{-1}\in G\mid \alpha\in U\}.
\]
The order of $I(G)$ as an inverse semigroup coincides with the order defined by inclusion.
A pair $U,V\in I(G)$ is said to be compatible if $U^{-1}V$ and $UV^{-1}$ belongs to $E(I(G))$.
If $U,V\in I(G)$ are compatible,
$U\cup V$ is an element of $I(G)$.
Note that $U\cup V$ is the least upper bound of $\{U,V\}$.
Thus $I(G)$ admits joins of compatible pairs in $I(G)$.
A subsemigroup $T\subset I(G)$ is said to be join closed if all joins of compatible pair of $T$ also belongs to $T$.

For $U\in I(G)$,
we have a homeomorphism $\rho_U\colon d(U)\to r(U)$ defined by
\[
\rho_U(d(\alpha))=r(\alpha), \alpha\in U.
\]
Then the map $U\mapsto \rho_U$ defines an action $\rho\colon I(G)\curvearrowright G^{(0)}$.
One can see that $\rho$ is strongly tight.
The following theorem is essentially same as \cite[Theorem 2.8]{matsnev_resende_2010}.

\begin{thm}[c.f.\ {\cite[Theorem 2.8]{matsnev_resende_2010}}] \label{theorem about the structure of ample groupoid}
	Let $G$ be an ample \'etale groupoid.
	Then $G$ is isomorphic to $I(G)\ltimes_{\rho}G^{(0)}$.
\end{thm}

\begin{proof}
For $\alpha\in G$,
there exists $U_{\alpha}\in I(G)$ such that $\alpha\in U_{\alpha}$ since $G$ is ample.
Then $[U_{\alpha}, d(\alpha)]\in I(G)\ltimes_{\rho}G^{(0)}$ is independent of the choice of $U_{\alpha}$.
Thus we obtain the map
\[
\Phi\colon G\ni \alpha\mapsto [U_{\alpha},d(\alpha)]\in I(G)\ltimes_{\rho}G^{(0)}.
\]
One can see that $\Phi$ is an isomorphism as a morphism between \'etale groupoids. 
Indeed, the map $\Psi\colon I(G)\ltimes_{\rho}G^{(0)}\to G$ defined by $\Psi([U,x])= d_U^{-1}(x)$ is the inverse map of $\Phi$.
\qed
\end{proof}

\begin{lem}\label{lemma about theta join closed is equiv to join closed}
Let $G$ be an ample \'etale groupoid.
Then a wide subsemigroup $T\subset I(G)$ is $\rho$-join closed if and only if $T$ is join closed.
\end{lem}

\begin{proof}
Assume that $T\subset I(G)$ is join closed.
Take $U\in I(G)$ and there exists a finite set $\mathcal{F}\subset E(I(G))$ such that $U \mathcal{F}\in T$ and $D_{U^*U}^{\rho}\subset \bigcup_{O\in \mathcal{F}}D_O^{\rho}$.
Observe that elements in $U\mathcal{F}$ are pairwisely compatible and $\bigvee_{O\in \mathcal{F}} UO=U$ holds.
Since $T$ is join closed,
$U$ belongs to $T$.
	
To show the converse,
assume that $T\subset I(G)$ is $\rho$-join closed.
Let $U,V\in T$ be compatible.
Put $\mathcal{F}\defeq \{U^{-1}U, V^{-1}V\}\subset E(I(G))$.
Then one can see that $(U\cup V) \mathcal{F}=\{U,V\}\subset T$ hold.
Since we have \[
\dom(\rho_{U\cup V})=d(U)\cup d(V)=\dom \rho_U\cup \dom\rho_V,\]
we obtain $U\cup V\in T$ by the $\rho$-closedness of $T$.
\qed
\end{proof}

Theorem \ref{main theorem},
Theorem \ref{theorem about the structure of ample groupoid} and Lemma \ref{lemma about theta join closed is equiv to join closed} yield the next corollary.

\begin{cor} \label{Corollary about inverse semigroups of compact bisections}
	Let $G$ be an ample \'etale groupoid.
	Then there is a correspondence between the set of all open wide subgroupoids of $G$ and the set of all wide join closed subsemigroups of $I(G)$.
\end{cor}

\subsection{Polycyclic monoids} \label{subsection of Polycyclic monoids}
 
We apply Theorem \ref{main theorem} to the polycyclic monoids $P_n$.

\begin{defi}
	Let $n\geq 2$ be a natural number.
	The polycyclic monoid $P_n$ is an inverse monoid defined by
	\[
	P_n\defeq \i<s_1,s_2,\dots,s_n \mid s_i^*s_j=\delta_{i,j} 1>.
	\]
\end{defi}

See \cite{Polyciclic} for details on the polycyclic monoids.

Set $\Sigma_n\defeq \{1,2,\dots,n\}$ and
\[\Sigma_n^{\N}\defeq\{(x_i)_{i=1}^{\infty}\mid x_i\in \Sigma_n \text{ for all $i\in\N$}\}.\]
It follows that $\Sigma_n^{\N}$ is a compact Hausdorff space from Tychonoff's theorem.
We write a finite sequence on $\Sigma_n$ like $\mu=(\mu_1,\mu_2,\dots,\mu_l)$,
where each $\mu_j$ is an element of $\Sigma_{n}$.
Here, $l\in\N$ is called the length of $\mu$, which we denote by $\lvert \mu\rvert$.
The only element of length zero is denoted by $\varepsilon$,
which is called the empty word.
We denote the set of all finite sequence on $\Sigma_n$ by $\Sigma_n^*$.
For $\mu\in \Sigma_n^*$,
we define a cylinder set $C(\mu)\subset \Sigma_n^{\N}$ as the set of all infinite sequences which begin with $\mu$ :
\[
C(\mu)\defeq \{(x_i)_{i=1}^{\infty}\in \Sigma_n^{\N}\mid x_i=\mu_i \text{ for all $i=1,2,\dots,\lvert\mu\rvert$}\}.
\]
We represent an element of $C(\mu)$ as $\mu x$ with $x\in \Sigma_n^{\N}$.
Each $C(\mu)$ is a compact open set of $\Sigma_n^{\N}$ and the family of all $C(\mu)$ is a basis of $\Sigma_n^{\N}$.
For $\mu\in\Sigma_n^*$,
we define $s_{\mu}\in P_n$ as
\[
s_{\mu}\defeq s_{\mu_1}s_{\mu_2}\cdots s_{\mu_{\lvert\mu\rvert}}.
\]
For the empty word $\varepsilon\in\Sigma_n^{\N}$,
we define $s_{\varepsilon}=1$.
It is known that an element of $P_n\setminus\{0\}$ is represented as the form $s_{\mu}s_{\nu}^*$ for unique $\mu,\nu\in\Sigma_n^*$.

Now we define an action $\beta\colon P_n\curvearrowright \Sigma_n^{\N}$.
For $s_{\mu}s_{\nu}^*\in P_n$,
define $\beta_{s_{\mu}s_{\nu}^*}\colon C(\mu)\to C(\nu)$ by
\[
\beta_{s_{\mu}s_{\nu}^*}(\nu x)=\mu x , x\in\Sigma_n^{\N}.
\]
Then the map $s_{\mu}s_{\nu}^*\mapsto \beta_{s_{\mu}s_{\nu}^*}$ defines an action $\beta\colon P_n\curvearrowright \Sigma_n^{\N}$.
Since the domain of $s_{\mu}s_{\mu}^*$ coincides with $C(\mu)$,
the action $\beta$ is strongly tight.

For $k,l\in\N$,
we define
\[
P_n^{k,l}\defeq\{s_{\mu}s_{\nu}^*\in P_n\mid\lvert\mu\rvert=k,\lvert\nu\rvert=l\}.
\]
Observe that
\[
M_n\defeq \bigcup_{k\in\N}P_n^{k,k}\cup\{0\}=\{s_{\mu}s_{\nu}^*\in P_n \mid \lvert{\mu}\rvert=\lvert\nu\rvert\}\cup\{0\}
\]
is a wide subsemigroup of $P_n$.

We investigate $\beta$-join closed subsemigroups $T\subset P_n$ such that $M_n\subset T$.
For $m\in \N$,
define
\[
P_n^m\defeq \bigcup_{k-l\in m\Z} P_n^{k,l}\cup\{0\}.
\]

Then one can see that $P_n^m$ is a $\beta$-join closed subsemigroup which contains $M_n$.
Notice that $P_n^{0}=M_n$.
Conversely, we obtain the following proposition.
\begin{prop}\label{inverse subsemigroup of polycyclic monoid}
Assume that $T\subsetneq P_n$ is a $\beta$-join closed subsemigroup which contains $M_n$.
Then $T=P_n^m$ holds for some $m\in \N$.
\end{prop}

In order to prove this proposition,
we prepare a few lemmas.
The next lemma follows from straightforward calculations.

\begin{lem}\label{lemma of polycyclic1}
For $i,j,k,l\in \N$,
we have
\begin{align*}
P_n^{i,j}P_n^{k,l}=
\begin{cases}
P_n^{i+k-j,l}\cup\{0\} & (k\geq j), \\
P_n^{i,j-k+l}\cup\{0\} & (k\leq j).
\end{cases}
\end{align*}

\end{lem}

\begin{lem}\label{lemma of polycyclic2}
	Let $T\subset P_n$ be a wide subsemigroup which contains $M_n$.
	Then the following statements hold :
	\begin{enumerate}
		\item If $s_{\mu}s_{\nu}^*\in T$ holds, then $P_n^{\lvert\mu\rvert,\lvert\nu\rvert}\subset T$ holds.
		\item $P_n^{k,l}\subset T$ implies $P_n^{l,k}\subset T$.
		\item $P_n^{k,l}\subset T$ implies $P_n^{k+1,l+1}\subset T$.
	\end{enumerate}
	Moreover, if $T$ is $\beta$-join closed,
	then the following holds :
	\begin{enumerate}
		\item[(4)] If $P_n^{k,l}$ holds for $k,l\in\Z_{>0}$,
		then $P_n^{k-1,l-1}\subset T$ holds.
	\end{enumerate}

\end{lem}

\begin{proof}
	(1) Assume $\lvert\mu\rvert=\lvert\mu'\rvert$ and $\lvert\nu\rvert=\lvert\nu'\rvert$ hold for $\mu',\nu'\in\Sigma_n^{\N}$.
	Then we have $s_{\mu'}s_{\mu}^*, s_{\nu}s_{\nu'}^*\in M_n\subset T$.
	Since we assume $s_{\mu}s_{\nu}^*\in T$,
	it follows
	\[
	s_{\mu'}s_{\nu'}=s_{\mu'}s_{\mu}^*s_{\mu}s_{\nu}^*s_{\nu}s_{\nu'}^*\in T.
	\]
	Hence we have $P_n^{\lvert\mu\rvert,\lvert\nu\rvert}\subset T$.
	
	(2) is clear,
	so we show (3) next.
	Take $s_{\mu}s_{\nu}^*\in P_n^{k,l}$ and $x,y\in \Sigma_n$ arbitrarily.
	Then we have
	\[s_{\mu x}s_{\nu x}^*=s_{\mu }s_{\nu}^* s_{\nu x}s_{\nu x}^*\in T,\]
	where we use the fact $s_{\nu x}s_{\nu x}^*\in M_n\subset T$.
	Using (1) in the above,
	we obtain $P_n^{k+1.l+1}\subset T$.
	
	Finally we show (4) under the assumption that $T$ is $\beta$-join closed.
	Take $s_{\mu}s_{\nu}^*\in P_n^{k-1,l-1}$.
	For each $x\in\Sigma_n$,
	we have
	\[
	s_{\mu}s_{\nu}^*s_{\nu x}s_{\nu x}^*=s_{\mu x}s_{\nu x}^*\in T,
	\]
	since we assume  $P_n^{k,l}\subset T$.
	Observe that
	\[
	D^{\beta}_{(s_{\mu}s_{\nu}^*)^*s_{\mu}s_{\nu}^* }=D^{\beta}_{s_{\nu}s_{\nu}^*}=\bigcup_{x\in \Sigma_n}D^{\beta}_{s_{\nu x}s_{\nu x}^*}(=C(\nu)).
	\]
	Since $T$ is $\beta$-join closed,
	we have $s_{\mu}s_{\nu}^*\in T$.
	Hence we have shown $P_n^{k-1,l-1}\subset T$.
	\qed
\end{proof}

\begin{proof}[\sc{Proof of Proposition \ref{inverse subsemigroup of polycyclic monoid}}]
	We may assume that $M_n\subsetneq T$.
	We define
	\[
	m\defeq \min \{\left\lvert\lvert\mu\rvert- \lvert\nu\rvert \right\rvert\in\N\mid  s_{\mu}s_{\nu}^*\in T\setminus M_n \}(>0).
	\]
	We show $T=P_n^m$.
	By the definition of $m$,
	there exists $s_{\mu}s_{\nu}^*\in T$ such that $\lvert\lvert\mu\rvert-\lvert\nu\rvert\rvert=m$.
	Since $T$ is closed under the inverse,
	we may assume $\lvert\mu\rvert-\lvert\nu\rvert=m$.
	Using (1) of Lemma \ref{lemma of polycyclic2},
	we have $P_m^{\lvert\mu\rvert,\lvert\nu\rvert}\subset T$.
	Applying (4) of Lemma \ref{lemma of polycyclic2} repeatedly,
	we obtain $P_n^{m,0}\subset T$ and it follows $P_n^{0,m}\subset T$ from (2) of Lemma \ref{lemma of polycyclic2}.
	Now one can see that $P_n^{k,l}\subset T$ holds for $k,l$ with $k-l\in m\Z$.
	Hence we obtain $P_n^m\subset T$.
	
	Next we show $T\subset P_n^m$.
	Assume that there exists $s_{\mu}s_{\nu}^*\in T$ such that $s_{\mu}s_{\nu}^*\not\in P_n^{m}$.
	We may assume that $\lvert\mu\rvert>\lvert\nu\rvert$.
	Take $a,b\in\N$ such that $\lvert\mu\rvert-\lvert\nu\rvert=am+b$ and $1\leq b\leq m-1$.
	We have $P_n^{\lvert\mu\rvert,\lvert\nu\rvert}\subset T$ by (1) of Lemma \ref{lemma of polycyclic2}.
	Using (4) of Lemma \ref{lemma of polycyclic2} repeatedly,
	we have $P_n^{am+b,0}\subset T$.
	Since we have $P_n^{m,0}\subset T$,
	it follows
	\[P_n^{(a-1)m+b,0}\subset P_n^{am+b,0}P_n^{m,0}\subset T,\]
	where we used Lemma \ref{lemma of polycyclic1}.
	Repeating this argument inductively,
	we obtain $P_n^{b,0}\subset T$.
	This contradicts to the minimality of $m$.
	Now we have shown $T=P_n^m$.
	\qed
\end{proof}

By Theorem \ref{main theorem} and Proposition \ref{inverse subsemigroup of polycyclic monoid},
an open proper intermediate subgroupoid between $P_n\ltimes_{\alpha}\Sigma_n^{\N}$ and $M_n\ltimes_{\beta}\Sigma_n^{\N}$ is given by the form $P_n^m\ltimes_{\beta}\Sigma_n^{\N}$ for some $m\in\N$.
Now we see $P_n^m\ltimes_{\beta}\Sigma_n^{\N}$ is closed. 
Observe that $P_n\ltimes_{\beta}\Sigma_n^{\N}$ has a continuous cocycle $c\colon P_n\ltimes_{\beta} \Sigma_n^{\N}\to\Z$ defined by $c([s_{\mu}s_{\nu}^*,x])=\lvert\mu\rvert-\lvert\nu\rvert$.
Since we have $P_n^m\ltimes_{\beta}\Sigma_n^{\N}=c^{-1}(m\Z)$,
it follows that $P_n^m\ltimes_{\beta}\Sigma_n^{\N}$ is a closed subset of $P_n\ltimes_{\beta}\Sigma_n^{\N}$.
Hence we obtain the next proposition.
\begin{prop}\label{closedness of subsemigroup of Pn}
Every open wide normal subgroupoid of $P_n\ltimes_{\beta}\Sigma_n^{\N}$ which contains $M_n\ltimes_{\beta}\Sigma_m^{\N}$ is closed.
\end{prop}

It follows from Corollary \ref{cor about compact strongly tight action} that $P_n\ltimes_{\beta}\Sigma_{n}^{\N}$ is isomorphic to the tight groupoid of $P_n$.
See Section 5 for more details.

%%%%%%%%%%%ここまで読んだ7/20

\section{Applications to the theory of C*-algebras}

\subsection{Analysis of Cartan intermediate subalgebras by using inverse semigroups}
In this section,
we explain a correspondence between Cartan intermediate subalgebras and certain subsemigroups of an inverse semigroup.

\begin{defi}
	Let $A$ be a C*-algebra.
	A commutative subalgebra $D\subset A$ is called a Cartan subalgebra if the following conditions hold :
	\begin{enumerate}
		\item The inclusion $D\subset A$ is nondegenerate (i.e.\ $D$ contains an approximate unit for $A$).
		\item The set of normalizers generates $A$, where $n\in A$ is called a normalizer if $nDn^*\cup n^*Dn\subset D$ holds.
		\item There is a faithful conditional expectation $E\colon A\to D$.
		\item The commutant $D'$ coincides with  $D$,
		where $D'\defeq\bigcap_{d\in D}\{a\in A\mid da=ad\}$.
	\end{enumerate}
	We call $(A,D)$ a Cartan pair.
	An intermediate subalgebra $D\subset B \subset A$ is called a Cartan intermediate subalgebra if $D$ is Cartan in $B$.
	
\end{defi}
	
	%Let $(A,D)$ be a Cartan pair.
	%Renault's cerebrated theorem in \cite{renault} states that $A$ is isomorphic to $C^*_{\lambda}(\Sigma)$ via 

	Renault's cerebrated theorem states that a Cartan pair arises from a twisted groupoid.
	We refer to \cite{Brown:2019aa} and \cite{asims} for twists of \'etale groupoids.
	A twisted groupoid over $G$ is a topological groupoid $\Sigma$ with the central extension
	\[
	G^{(0)}\times \T \hookrightarrow \Sigma \overset{q}{\twoheadrightarrow} G,
	\]
	where $\T$ is the circle group.
	In this paper, this twist is abbreviated to $q\colon \Sigma\to G$.
	We denote the reduced C*-algebra of the twist $q\colon \Sigma\to G$ by $C^*_{\lambda}(\Sigma)$.
	Recall that $C^*_{\lambda}(\Sigma)$ contains $C_0(G^{(0)})$ as a subalgebra.
	We denote the reduced C*-algebra of $G$ by $C^*_{\lambda}(G)$,
	which is isomorphic to the reduced C*-algebra of the trivial twist $G\times \T\to G$.
	
	\begin{thm}[{\cite[Theorem 5.9]{renault}}]
		Let $(A,D)$ be a Cartan pair where $A$ is separable.
		Then there exists a twist $q\colon \Sigma\to G$ such that $A$ is isomorphic to $C^*_{\lambda}(\Sigma)$ via an isomorphism which maps $D$ to $C_0(G^{(0)})$,
		where $G$ is second countable topologically principal locally compact Hausdorff \'etale groupoid.
		This twist $q\colon \Sigma\to G$ is unique up to isomorphism.
	\end{thm}

	From now on,
	we identify $C^*_{\lambda}(\Sigma)$ and $C_0(G^{(0)})$ with $A$ and $D$ respectively for a Cartan pair $(A,D)$.
	
	Let $q\colon \Sigma\to G$ be a twist and $H\subset G$ be a wide open subgroupoid.
 	\cite[Lemma 3.2]{Brown:2019aa} states that $\Sigma_H\defeq q^{-1}(H)$ naturally becomes a twist over $H$ and there exists a natural inclusion $C^*_{\lambda}(\Sigma_H)\subset C^*_{\lambda}(\Sigma)$.
	The authors in \cite{Brown:2019aa} showed this map $H\mapsto C^*_{\lambda}(\Sigma_H)$ gives a certain correspondence as follows.

	\begin{thm}[{\cite[Theorem 3.3, Lemma 3.4]{Brown:2019aa}}] \label{one to one corr between C* and groupoid}
	Let $(A, D)$ be a Cartan pair with a separable $A$ and $q\colon \Sigma\to G$ be an associated twist.
	Then the above map $H\mapsto C^*_{\lambda}(\Sigma_H)$ gives a one-to-one correspondence between the set of open wide subgroupoids of $G$ and the set of Cartan intermediate subalgebras $D\subset B\subset A$.
	In addition,
	there exists a conditional expectation from $C^*_{\lambda}(\Sigma)$ to $C^*_{\lambda}(\Sigma_H)$ if and only if $H\subset G$ is closed.
	\end{thm}

	Combining Theorem \ref{one to one corr between C* and groupoid} with Theorem \ref{main theorem},
	we obtain the next Corollary.

	\begin{cor}
		Let $(A,D)$ be a Cartan pair with separable $A$ and $q\colon \Sigma\to G$ be an associated twist.
		Assume that $G=S\ltimes_{\alpha}X$ holds for some strongly tight action $\alpha\colon S\curvearrowright X$.
		Then there exists a one-to-one correspondence between the set of $\alpha$-join closed wide subsemigroups of $S$ and the set of Cartan intermediate subalgebras $D\subset B\subset A$. 
		More precisely,
		the map $T\mapsto C^*_{\lambda}(\Sigma_{T\ltimes_{\alpha}X})$ gives the above correspondence.
	\end{cor}
	
	\begin{ex}
		We investigate certain subalgebras of the Cuntz algebras by using the polycyclic monoids here.
		For $n\in\N$ with $n\geq 2$,
		the Cuntz algebra $O_n$ is the universal unital C*-algebra generated by isometries $S_1,\dots,S_n$ which satisfy Cuntz relation as follows : 
		\[
		S_i^*S_j=\delta_{i,j}1, \sum_{i=1}^nS_iS_i^*=1.
		\]
		For a finite sequence $\mu=(\mu_1,\dots,\mu_l)$ on $\{1,\dots,n\}$,
		we define 
		\[S_{\mu}\defeq S_{\mu_1}S_{\mu_2}\cdots S_{\mu_l}.\]
		Then $O_n$ is the closure of the linear span of $\{S_{\mu}S_{\nu}^*\}_{\mu,\nu}$,
		where $\mu$ and $\nu$ are taken over the all finite sequences on $\{1,\dots,n\}$.
		Let $D_n$ be the subalgebra of $O_n$ generated by $\{S_{\mu}S_{\mu}^*\}_{\mu}$,
		where $\mu$ is taken over the all finite sequences on $\{1,\dots,n\}$.
		We denote the gauge action by $\tau\colon\T\curvearrowright O_n$.
		Note that the gauge action satisfies $\tau_z(S_i)=zS_i$ for all $z\in \T$ and $i=1,2,\dots,n$.
		We denote the fixed point algebra of $\tau$ by
		\[O_{n}^{\tau}\defeq \bigcap_{z\in\T}\{x\in O_n\mid \tau_z(x)=x\}.\]
		Then $O_{n}^{\tau}$ is the closure of the linear span of
		\[
		\{S_{\mu}S_{\nu}^*\in O_n\mid \lvert\mu\rvert=\lvert\nu\rvert\},
		\]
		where $\lvert\mu\rvert$ denotes the length of $\mu$.

		The polycyclic monoids have strongly tight actions $\beta\colon P_n\curvearrowright \Sigma_n^{\N}$,
		described in subsection \ref{subsection of Polycyclic monoids}.
		Put $G_n\defeq P_n\ltimes_{\beta}\Sigma_n^{\N}$.
		Then $G_n$ is a topologically principal locally compact Hausdorff second countable ample groupoid.
		For $s_i\in P_n$,
		let $\chi_{[s_i,D_{s_i^*s_i}]}$ denote the characteristic function on $[s_i,D_{s_i^*s_i}]\subset G_n$.
		Then $\{\chi_{[s_i,D_{s_i^*s_i}]}\}_{i=1}^n$ are elements of $C^*_{\lambda}(G_n)$ and generate $C^*_{\lambda}(G_n)$.
		Since $\{\chi_{[s_i,D_{s_i^*s_i}]}\}_{i=1}^n$ satisfies the Cuntz relation,
		$O_n$ and $C^*_{\lambda}(G_n)$ are isomorphic via the unique isomorphism $\Phi\colon O_n\to C^*_{\lambda}(G_n)$ such that $\Phi(S_i)=\chi_{[s_i,D_{s_i^*s_i}]}$ holds for all $i=1,\dots,n$.		One can see that 
		\[\Phi(D_n)=C(\Sigma_n^{\N}) \text{ and } \Phi(O_n^{\tau})= C^*_{\lambda}(M_n\ltimes_{\beta}\Sigma_n^{\N})\]
		hold.
		Define $O_n^m\subset O_n$ to be the subalgebra generated by
		\[
		\{S_{\mu}S_{\nu}^*\in O_n\mid \lvert\mu\rvert-\lvert\nu\rvert\in m\Z\}.
		\]
		One can see that \[\Phi(O_n^m)=C^*_{\lambda}(P_n^m\ltimes_{\beta}\Sigma_n^{\N})\]
		holds.
		Therefore it follows from Proposition \ref{inverse subsemigroup of polycyclic monoid} that a Cartan intermediate subalgebra $O_n^{\tau}\subset B\subset O_n$ coincides with $O_n^m$ for some $m\in \N$.
		Moreover,
		every Cartan intermediate subalgebra between $O_n^{\tau}$ and $O_n$ admits a conditional expectation from $O_n$ by Proposition \ref{closedness of subsemigroup of Pn} and Theorem \ref{one to one corr between C* and groupoid}.
		
		We note that $O_n^m$ is isomorphic to $O_{n^m}$.
		Indeed, $\{S_{\mu}\}_{\lvert\mu\rvert =m}$ generates $O_n^m$ and satisfies the Cuntz relation.
	\end{ex}

%%%%%%%%%%%ここまで読んだ7/21

\section{Relation between strongly tight actions and tight groupoids}\label{section about characterization of strongly tight action}

In this section,
we observe that tight groupoids, which are investigated in \cite{ExelPardo2016}, are related with strongly tight actions.

\subsection{Tight groupoids}

First we recall the definition of tight groupoids.
Refer to \cite{ExelcombinatrialC*algebra} or \cite{ExelPardo2016} for more details.
Let $S$ be an inverse semigroup.
A character on $E(S)$ is a nonzero semigroup homomorphism from $E(S)$ to $\{0,1\}$,
where $\{0,1\}$ is equipped with the usual multiplication.
We denote the set of all characters on $E(S)$ by $\widehat{E}(S)$.
Letting $\widehat{E}(S)$ be equipped with the pointwise convergence topology,
$\widehat{E}(S)$ is a locally compact Hausdorff space.
For a $\xi\in\widehat{E}(S)$,
$\xi^{-1}(\{1\})\subset E(S)$ is a proper filter in the following sense :
\begin{enumerate}
\item $\xi^{-1}(\{1\})$ does not contain $0$,
\item 	if $e$ and $f$ belongs to $\xi^{-1}(\{1\})$, then $ef$ also belongs to $\xi^{-1}(\{1\})$,
\item  if $e\in\xi^{-1}(\{1\})$ and $f\geq e$ hold, then $f$ belongs to $\xi^{-1}(\{1\})$.
\end{enumerate}

A character $\xi\in \widehat{E}(S)$ is called an ultracharacter if $\xi^{-1}(\{1\})$ is a maximal proper filter.
A character $\xi\in \widehat{E}(S)$ is an ultracharacter if and only if there is no character $\eta\in\widehat{E}(S)$ such that $\xi<\eta$ holds.
The set of all ultracharacters on $E(S)$ is denoted by $\widehat{E}_{\infty}(S)$.
The closure of $\widehat{E}_{\infty}(S)$ in $\widehat{E}(S)$ is denoted by $\widehat{E}_{\tight}(S)$.
An element in $\widehat{E}_{\tight}(S)$ is called a tight character.

We define the spectral action $\beta\colon S\curvearrowright \widehat{E}(S)$.
For $e\in E(S)$,
put \[ D_e^{\beta}\defeq\{\xi\in\widehat{E}(S)\mid \xi(e)=1\}.\]
Note that $D_e^{\beta}$ is a compact open set of $\widehat{E}(S)$.
For $s\in S$ and $\xi\in D_{s^*s}^{\beta}$,
define $\beta_s(\xi)\in D_{ss^*}^{\beta}$ by
\[
\beta_s(\xi)(e)\defeq \xi(s^*es), e\in E(S).
\]
Then $\beta_s\colon D_{s^*s}^{\beta}\to D_{ss^*}^{\beta}$ is a homeomorphism.
The map $s\mapsto \beta_s$ defines an action $\beta\colon S\curvearrowright \widehat{E}(S)$.
It is known that $\widehat{E}_{\infty}(S)$ and $\widehat{E}_{\tight}(S)$ are $\beta$-invariant (see \cite[Proposition 12.11]{ExelcombinatrialC*algebra}).
The restrictions of $\beta$ to $\widehat{E}_{\infty}(S)$ and $\widehat{E}_{\tight}(S)$ are denoted by
\[
\theta_{\infty}\colon S\curvearrowright \widehat{E}_{\infty}(S) \text{ and } \theta\colon S\curvearrowright \widehat{E}_{\tight}(S)
\]
respectively.
The tight groupoid of $S$ is defined as $G_{\tight}(S)\defeq S\ltimes_{\theta} \widehat{E}_{\tight}(S)$.

\subsection{Characterization of tight groupoids}

Strongly tight actions with nonempty domains are characterized as the following theorem.

\begin{thm}\label{characterization of strongly tight action}
	Let $S$ be an inverse semigroup,
	$X$ be a locally compact Hausdorff space and $\alpha\colon S\curvearrowright X$ be a strongly tight action such that $D_e\not=\emptyset$ holds for each $e\in E(S)\setminus\{0\}$.
	Then there exists a homeomorphism  
	\[X\ni x\mapsto \xi_x\in\widehat{E}_{\infty}(S)\]
	which induces an isomorphism
	\[S\ltimes_{\alpha}X\ni [s,x]\mapsto [s,\xi_x]\in S\ltimes_{\theta_{\infty}}\widehat{E}_{\infty}(S).\]
\end{thm}

\begin{proof}
	This is a simple modification of \cite[Proposition 5.5]{STEINBERG2010689}.
	We give a proof for the reader's convenience.
	
	For $x\in X$,
	we define $\xi_x\in\widehat{E}(S)$ by
	\begin{align*}
	\xi_x(e)\defeq
	\begin{cases}
	1 & (x\in D^{\alpha}_e), \\
	0 & (x\not\in D^{\alpha}_e).
	\end{cases}
	\end{align*}
	We show $\xi_x\in\widehat{E}_{\infty}(S)$.
	Assume that there exists $\eta\in\widehat{E}(S)$ such that $\xi_x< \eta$.
	Then there exists $f\in E(S)$ such that $\xi_x(f)=0$ and $\eta(f)=1$.
	Since we assume that $\{D^{\alpha}_e\}_{e\in E(S)}$ is a basis of $X$,
	there exists $e\in E(S)$ such that $x\in D^{\alpha}_e$ and $D^{\alpha}_e\cap D^{\alpha}_f=\emptyset$.
	By the assumption $\xi_x< \eta$, we have $\eta(e)=1$.
	By $D^{\alpha}_{ef}=D^{\alpha}_e\cap D^{\alpha}_f=\emptyset$,
	we have $ef=0$ and therefore $\eta(ef)=0$.
	This contradicts to $\eta(ef)=\eta(e)\eta(f)=1$.
	Hence $\xi_x$ is an ultracharacter.
	
	We define the map $\Phi\colon X\ni x\mapsto \xi_x\in\widehat{E}_{\infty}(S)$.
	We show that $\Phi$ is a homeomorphism. 
	It is easy to show that $\Phi$ is continuous.
	To show that $\Phi$ is injective,
	take $x,y\in X$ with $x\not=y$.
	Since a family $\{D^{\alpha}_e\}_{e\in E(S)}$ is a basis of $X$,
	there exists $e\in E(S)$ such that $x\in D_e^{\alpha}$ and $y\not\in D_e^{\alpha}$.
	Then $\xi_x(e)=1$ and $\xi_y(e)=0$.
	Therefore we have $\xi_x\not=\xi_y$ and $\Phi$ is injective.
	
	Next, take $\xi\in\widehat{E}_{\infty}(S)$ to show that $\Phi$ is surjective.
	Because a family $\{D^{\alpha}_e\mid \xi(e)=1\}$ has the finite intersection property,
	$\bigcap_{\xi(e)=1}D^{\alpha}_e$ is not empty.
	Take $x\in\bigcap_{\xi(e)=1}D^{\alpha}_e$.
	Then we have $\xi\leq \xi_x$.
	By the maximality of $\xi$,
	we obtain $\xi=\xi_x$.
	Therefore the map $x\mapsto \xi_x$ is surjective.
	
	Now one can check that $\Phi(D_e^{\alpha})=D_e^{\beta}$ holds.
	Using this, it follows that $\Phi$ is a homeomorphism.
	
	It is straightforward to check that there exists a (unique) isomorphism which maps
	$[s,x]\in S\ltimes_{\alpha}X$ to $[s,\xi_x]\in S\ltimes_{\theta_{\infty}}\widehat{E}_{\infty}(S)$.
	\qed
\end{proof}

%%%%%%%%%%%%%ここまで読んだ 7/22食事前

\begin{rem}
	It seems difficult to drop the assumption that $D_e\not=\emptyset$ for $e\in E(S)\setminus \{0\}$.
	Define matrices
	\begin{align*}
	p=
	\begin{pmatrix}
	1 & 0 \\
	0 & 0
	\end{pmatrix},
	q=
	\begin{pmatrix}
	0 & 0 \\
	0 & 1
	\end{pmatrix}
	.
	\end{align*}
	Then $E\defeq\{0,p,q,1\} $ is a semilattice with respect to the usual multiplication.
	Let $X=\{x\}$ be a singleton.
	Define an action $\alpha\colon E\curvearrowright X$ by declaring $D_1=X$ and  $D_p=D_q=D_0=\emptyset$.
	Then $\alpha$ is strongly tight,
	although $\widehat{E}_{\infty}$ is not homeomorphic to $X$.
	Note that $\xi_x$, which is defined in the proof of Theorem \ref{characterization of strongly tight action}, is not an ultracharacter.
	Therefore it seems difficult to find a natural map between $X$ and $\widehat{E}_{\infty}$.
\end{rem}

The author in \cite{Lawson:2010aa} showed the following theorem.
\begin{thm}[{\cite[Theorem 2.5]{Lawson:2010aa}}]\label{compactable condition}
	Let $E$ be a semilattice with zero and unit elements.
	Then $\widehat{E}_{\infty}=\widehat{E}_{\tight}$ holds if and only if $\widehat{E}_{\infty}$ is compact.
\end{thm}

Theorem \ref{characterization of strongly tight action} and Theorem \ref{compactable condition} yield the following characterization of tight groupoids.

\begin{cor}\label{cor about compact strongly tight action}
	Let $S$ be an inverse semigroup.
	Consider the following conditions.
	\begin{enumerate}
		\item $S$ has a strongly tight action to a compact Hausdorff set $X$.
		\item $\widehat{E}(S)_{\infty}$ is compact,
		\item $\widehat{E}(S)_{\tight}=\widehat{E}(S)_{\infty}$.
	\end{enumerate}
	Then $(1) \Leftrightarrow (2)$ and $(2) \Rightarrow (3)$ hold.
	If $S$ has a unit element,
	$(3) \Rightarrow (2)$ also holds.
	Moreover,
	if (1) holds,
	then $S\ltimes X$ is isomorphic to $G_{\tight}(S)$.
\end{cor}

\begin{rem}\label{ultra filter is tight but not compact}
	The implication $(3)\Rightarrow(2)$ in Corollary \ref{cor about compact strongly tight action} dose not necessarily hold in general.
	Let $E$ be a semilattice generated by $0$ and $\{p_i\}_{i\in \N}$ with the relation
	\begin{align*}
	p_ip_j=
	\begin{cases}
	p_i & (i=j),\\
	0 & (i\not=j).
	\end{cases}
	\end{align*}
	Then $\widehat{E}_{\infty}=\widehat{E}_{\tight}$ holds,
	although $\widehat{E}_{\infty}$ is not compact.
	Indeed $\widehat{E}_{\infty}$ is homeomorphic to $\N$.
\end{rem}

\begin{rem}
	There exists a semilattice $E$ such that $\widehat{E}_{\infty}$ is a locally compact although $\widehat{E}_{\infty}\subsetneq \widehat{E}_{\tight}$ holds.
	Let $E$ be the semilattice in Remark \ref{ultra filter is tight but not compact}.
	Put $E^1\defeq E\cup\{1\}$.
	Then $\widehat{E^1}_{\infty}$ is locally compact.
	In addition we have $\widehat{E^1}_{\infty}\subsetneq \widehat{E^1}_{\tight}$.
	Indeed, $\widehat{E^1}_{\infty}$ and $\widehat{E^1}_{\tight}$ are isomorphic to $\N$ and $\N\cup\{\infty\}$ respectively.
	Therefore we can not relax the condition (2) in Corollary \ref{cor about compact strongly tight action}.
\end{rem}

\bibliographystyle{plain}
\bibliography{bunken}

\end{document}